\documentclass[conference]{IEEEtran}  

\usepackage{verbatim}
\usepackage{amsthm}
\usepackage{tikz}
\usepackage{subfig}
\usepackage{float}
\usepackage{url}
\usepackage{fancyhdr,amssymb}
\usepackage{algpseudocode}
\usepackage[T1]{fontenc}
\usepackage{mathtools}
\usepackage[noadjust]{cite}
\usepackage{dsfont}
\usepackage[english]{babel}
\usepackage{amsmath}
\usepackage{graphicx}
\usepackage{float}
\usepackage{fancyhdr,amssymb}
\usepackage[colorinlistoftodos]{todonotes}
\usepackage{enumerate}
\usepackage{mathrsfs}
\usepackage[utf8x]{inputenc}

\usetikzlibrary{fadings}
\usetikzlibrary{shapes.misc}
\tikzset{cross/.style={cross out, draw=black, minimum size=2*(#1-\pgflinewidth), inner sep=0pt, outer sep=0pt},
cross/.default={1pt}}
\usetikzlibrary{shapes,arrows}
\usetikzlibrary{positioning}
\usetikzlibrary{plotmarks}
\usepackage{pgfplots}
\pgfplotsset{compat=newest}
\pgfplotsset{plot coordinates/math parser=false}
\newlength\figureheight
\newlength\figurewidth
\usetikzlibrary{decorations,decorations.pathmorphing,decorations.pathreplacing}

\newtheorem{theorem}{Theorem}
\newtheorem{lemma}[theorem]{Lemma} 
\newtheorem{proposition}[theorem]{Proposition}

\newtheorem{definition}[theorem]{Definition}

\newtheorem{claim}[theorem]{Claim}

\newcommand\norm[1]{\left\lVert#1\right\rVert}
\newcommand{\pmat}[3]{\begin{pmatrix} #1 & #2 &\cdots & #3 \end{pmatrix}}
\newcommand{\ftf}{( v_1, \ldots, v_K )}
\newcommand{\moq}[3]{\eta(#1,#2,#3)}

{
      \theoremstyle{definition}

  }

\newcommand{\opt}{^\ast}
\newcommand{\transp}{^\top}

\newcommand{\A}{\textbf{a}}

\newcommand{\R}[1]{\mathbb{R}^{#1}}
\newcommand{\mat}[2]{\mathbb{R}^{{#1} \times {#2} } }
\newcommand{\posdef}[1]{\mathbb{R}^{#1 \times #1}_{++}}
\newcommand{\possemdef}[1]{\mathbb{R}^{#1 \times #1}_{+}}
\newcommand{\sphere}[1]{S_{n - 1}}
\newcommand{\eigval}[1]{\lambda_{#1}}

\newcommand{\kalman}[3]{ \mathcal{ C }_{ ( #1, #2, #3 ) } }
\newcommand{\grammian}[3]{G_{ ( #1, #2, #3 ) } }

\DeclareSymbolFont{symbolsC}{U}{pxsyc}{m}{n}
\DeclareMathSymbol{\coloneqq}{\mathrel}{symbolsC}{"42}\DeclareMathOperator{\trace}{tr}
\DeclareMathOperator{\image}{image}
\DeclareMathOperator{\Span}{span}

\DeclareMathOperator{\FF}{FF}
\DeclareMathOperator{\FP}{FP}
\DeclareMathOperator{\MFP}{NFP}

\DeclarePairedDelimiter{\abs}{\lvert}{\rvert}

\title{On a frame theoretic measure of quality of LTI systems 
}

\author{Mohammed Rayyan Sheriff, Debasish Chatterjee 
\thanks{D Chatterjee was supported in part by the grant 12IRCCSG005 from IRCC, IIT Bombay.}
\thanks{M Rayyan and Debasish Chatterjee are with the interdisciplinary program in Systems and Control Engineering, IIT Bombay, India. {\tt\small Emails: mohammedrayyan@sc.iitb.ac.in, dchatter@sc.iitb.ac.in }}
}

\begin{document}

\maketitle
\thispagestyle{empty}
\pagestyle{empty}

\begin{abstract}

It is of practical significance to define the notion of a measure of quality of a control system, i.e., a quantitative extension of the classical notion of controllability. In this article we demonstrate that the three standard measures of quality involving the trace, minimum eigenvalue, and the determinant of the controllability grammian achieve their optimum values when the columns of the controllability matrix from a tight frame. Motivated by this, and in view of some recent developments in frame theoretic signal processing, we provide a measure of quality for LTI systems based on a measure of tightness of the columns of the reachability matrix .

\end{abstract}

\section{Introduction}

\label{section:Review-of-some-conventional-MOQ's-for-LTI-systems}

Let us consider a linear time invariant controlled system
\begin{equation}
\label{eq:lti-system}
x(t + 1) = A x(t) + B u(t) \quad \text{for t = 1, 2, \ldots,}
\end{equation}
where \( A \in \mat{n}{n} \) and \( B \in  \mat{n}{m} \), with \( n \) and \( m \) some positive integers \footnote{The set of \( n \times n \) real symmetric positive definite and positive  semi definite matrices are denoted by \( \posdef{n} \) and \( \possemdef {n} \) respectively, \( \image(M)\) is the \emph{column space} of the matrix \( M \), \( \trace(M) \) is the \emph{trace} of the matrix \( M \), \( x\transp \) is the \emph{transpose} of the vector \( x \), \( \norm{x}_2 \) is the \( \ell_2 \)-\emph{norm} of the vector \( x \).}. For \( T \in \mathbb{N} \), let \( \R{n \times Tm} \ni \kalman{A}{B}{T} \coloneqq \pmat{B}{AB}{A^{T - 1}B} \). Then the \emph{reachable space} \( \mathcal{R}_{ (A,B) } \) and the \emph{reachability index}  \( \tau \) of the LTI system \eqref{eq:lti-system} are defined by
\begin{align*}
\mathcal{R}_{ (A,B) } & \coloneqq \image \big( \kalman{A}{B}{n} \big), \\
\tau & \coloneqq \min \big\{ t \in \mathbb{N} \vert \image \big( \kalman{A}{B}{t} \big) = \mathcal{R}_{ (A,B) } \big\}.
\end{align*}
We say that the control system \eqref{eq:lti-system} is \emph{controllable} in the classical sense if it admits a control sequence \( \big( u(0), \ldots, u(T - 1) \big) \) that can transfer the states of the system \eqref{eq:lti-system} from \( x(0) = x_0 \) to \( x(T) = x_T \) for preassigned values of \( x_0,x_T \) whenever \( T \geq \tau \). It is common knowledge that the LTI system \eqref{eq:lti-system} is controllable if and only if the matrix \( \kalman{A}{B}{\tau} \) is of rank \( n \). 

However, in most of the practical scenarios, information about the system behavior provided by the classical notion of controllability is limited. In particular, no assessment of ``how controllable'' is a given LTI system is provided by the classical ideas. Naturally, together with knowing whether a control system is controllable or not, one would also like to know \emph{how controllable} or \emph{how good} is the control system. In other words, there is a clear need to quantify the notion of controllability.

For instance, let us consider two second order LTI systems described by the pairs \( (A,B) \) and \( (A',B') \) such that for \( T = 3 \) the columns of the matrices \( \kalman{A}{B}{T} \), \( \kalman{A'}{B'}{T} \) are as shown in the following figure.
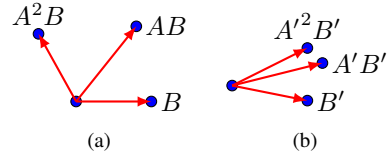
\begin{figure}[H]
\label{fig:lti-system-comparison}
			\centering
			\subfloat[]{
			\begin{tikzpicture}
			\coordinate (A) at (0,0);
			\coordinate (B) at (1,0);
			\coordinate (C) at (0.8,1);
            \coordinate (D) at (-0.5,0.9);
			\draw [fill=blue] (A) circle (2pt) node [left] {};
			\draw [fill=blue] (B) circle (2pt) node [right] {\( B \)};
			\draw [fill=blue] (C) circle (2pt) node [right] {\( AB \)};
			\draw [fill=blue] (D) circle (2pt) node [above] {\( A^2B \)};
            
			\draw [-latex, red, thick] (A) -- (B);
			\draw [-latex, red, thick] (A) -- (C);
            \draw [-latex, red, thick] (A) -- (D);
			\end{tikzpicture}
				}	
				\subfloat[]{
			\begin{tikzpicture}
			\coordinate (A) at (0,0);
			\coordinate (B) at (1,-0.2);
			\coordinate (C) at (1.2,0.3);
            \coordinate (D) at (1,0.5);
			\draw [fill=blue] (A) circle (2pt) node [left] {};
			\draw [fill=blue] (B) circle (2pt) node [right] {\( B' \)};
			\draw [fill=blue] (C) circle (2pt) node [right] {\( A'B' \)};
     		\draw [fill=blue] (D) circle (2pt) node [above] {\( {A'}^2B' \)};
            
			\draw [-latex, red, thick] (A) -- (B);
			\draw [-latex, red, thick] (A) -- (C);
            \draw [-latex, red, thick] (A) -- (D);
			\end{tikzpicture}
				}
				
	\caption{Comparison of the orientation of the columns of \( \kalman{A}{B}{T} \) and \( \kalman{A'}{B'}{T} \)}		
		\end{figure}
We see that for the system \( (A',B') \), since the vectors \( B', A'B', {A'}^2B' \) form small angles with each other, intuitively, the transfer of the system state to final states \( x_T \) that is in approximately orthogonal directions to these vectors will demand control with higher magnitudes. In contrast, for the system \( (A,B) \) the vectors are spread out in space and will in general require lower control magnitude. Such a collection of vectors are called as \emph{tight frames}. It turns out that tight frames have many practical advantages over generic collection of vectors \cite{sheriff2016optimal},  \cite{shen2006image}, \cite{zhou2016adaptive}, \cite{strohmer2003grassmannian}, \cite{holmes2004optimal} from signal processing perspective. In this article we shall establish that if the columns of \( \kalman{A}{B}{T} \) constitute a tight frame, one also gets several control theoretic advantages.

In particular, we provide a notion of quantitative controllability, which we shall refer to as the \emph{Measure Of Quality} (MOQ) of a control system. Intuitively speaking, any measure of quality should relate to some important characteristics of the system like average control energy / control effort required, robustness to input / noise, ability to control the systems with sparse controls etc We shall see below that such intuitive ideas are indeed justified.

Let us define the \emph{control effort} \( J : \R{mT} \longrightarrow \R{} \) of a control sequence \( \big( u(0), \ldots, u(T - 1) \big) \), defined by 
\[
J \big( u(0), \ldots, u(T - 1) \big) \coloneqq \sum_{t = 0}^{T -1} \norm{u(t)}_2^2 ,
\] 
is a quantity of practical significance. For example, the value of the control effort in an electronic circuit system would involve information of the amount of power drawn to control the system. One of the primary objectives in controlling a dynamical system efficiently could be to minimize the required control effort. Therefore, let us consider the following optimal control problem:
\begin{equation}
\label{eq:optimal-control-problem}
\begin{aligned}
& \underset{u(t)}{\text{minimize}}
& &  J \big( u(0), \ldots, u(T - 1) \big) \\
& \text{subject to}
& & x(0) = 0, \; x(T) = x, \\
& & & x(t) \text{ evolves according to \eqref{eq:lti-system}}.
\end{aligned}
\end{equation}
Let the \emph{grammian} \( \grammian{A}{B}{T} \) be defined by 
\[
\grammian{A}{B}{T} \coloneqq \sum_{t = 0}^{T -1} A^t B B^\top (A^\top)^t = \kalman{A}{B}{T} \kalman{A}{B}{T}\transp.
\] 
We know that the optimal control problem \eqref{eq:optimal-control-problem} admits an unique optimal control sequence \( \big( u\opt(0), \ldots, u\opt(T - 1) \big) \) with the optimum control effort \( J(A,B,x) \) given by:
\begin{equation}
\begin{aligned}
\big( { u\opt(T - 1) }\transp \cdots \; { u\opt(0) }\transp   \big)\transp &\coloneqq \kalman{A}{B}{T}^+ x, \\
J(A,B,x) \coloneqq J \big( u\opt(0), \ldots, u\opt(T - 1) \big)
& = x\transp \grammian{A}{B}{T}^{-1} x.
\end{aligned}
\end{equation}

We briefly review some of the most well known MOQ's studied in the literature \cite{muller1972analysis}, \cite{johnson1969optimization}, \cite{pasqualetti2014controllability}; they capture someinformation pertaining to  optimum control effort needed to control the system:
\begin{enumerate}[(i)]
\item \( \trace \big( \grammian{A}{B}{T}^{-1} \big) \): This is proportional to the average optimal control effort needed to transfer the system state from origin (i.e., \( x_0 = 0\),) to a random point that is uniformly distributed on the unit sphere. 

\item \( \lambda_{\min}^{-1} \big( \grammian{A}{B}{T} \big) \): This gives the maximum control effort needed to transfer the system state from origin to any point on the unit sphere.

\item \( \det \big( \grammian{A}{B}{T} \big) \): This proportional to the volume of the ellipsoid containing points to which the system state can be transferred to from the origin using at most unit control effort. 
\end{enumerate}

A finite value of the MOQ in (i) and (ii), and a nonzero value in (iii), implies that the system \eqref{eq:lti-system} is controllable in the classical sense. In addition, since these quantities also contain some information about the optimal control effort, they can be considered as valid candidates for an MOQ. However, these MOQ's are increasingly difficult to compute as the size of the system becomes larger, and since they arise primarily in the context of the optimal control problem \eqref{eq:optimal-control-problem}, they do not provide meaningful insights about properties such as immunity to noise, robustness, ability to generate sparse controls etc.

For the dynamical system \eqref{eq:lti-system}, we shall motivate the notion of an MOQ from the perspective of the orientation of the columns of the matrix \( \kalman{A}{B}{T} \). We shall prove in this article that the three MOQ's discussed above achieve their optimum value when the columns of the matrix \( \kalman{A}{B}{T} \) constitute a tight frame. We provide an MOQ for the system \eqref{eq:lti-system} as a measure of the tightness of the columns of \( \kalman{A}{B}{T} \), and provide a sufficient condition for the classical controllability of the system \eqref{eq:lti-system} in terms of the proposed MOQ.

\section{Review of tight frames}
\label{sec:tight-frame}
\begin{definition}
For an \( n \) dimensional Hilbert space  \( H_n \) with an inner product \( \langle \cdot , \cdot \rangle \), a sequence of vectors \( \ftf \subset H_n \) is said to be a \emph{frame} of \( H_n \) if it spans \( H_n \); a frame is said to be \emph{tight},  if there exists a number \( \A > 0 \) such that for every \( v \in H_n \), the equality \( \A \norm{v}^2 = \sum_{i = 1}^K \abs{\langle v, v_i \rangle}^2 \) holds.
\end{definition}
A classical example of a tight frame is an orthonormal basis of \( H_n \) that satisfies the definition with \( \A = 1 \) and \( K = n \). Some other  examples of the tight frames of \( \R{2} \) are shown in figure below.
\begin{figure}[H]
\label{fig:tight-frmaes-example}
			\centering
			\subfloat[]{
			\begin{tikzpicture}
			\coordinate (A) at (0,0);
			\coordinate (B) at (1,0);
			\coordinate (C) at (-0.5, -0.866);
            \coordinate (D) at (-0.5, 0.866);
			\draw [fill=blue] (A) circle (2pt) node [left] {};
			\draw [fill=blue] (B) circle (2pt) node [right] {};
			\draw [fill=blue] (C) circle (2pt) node [right] {};
			\draw [fill=blue] (D) circle (2pt) node [above] {};
            
			\draw [-latex, red, thick] (A) -- (B);
			\draw [-latex, red, thick] (A) -- (C);
            \draw [-latex, red, thick] (A) -- (D);
			\end{tikzpicture}
				}	
				\subfloat[]{
			\begin{tikzpicture}
			\coordinate (A) at (0,0);
			\coordinate (B) at (1,0);
			\coordinate (C) at (0.707,0.707);
            \coordinate (D) at (0,1);
            \coordinate (E) at (0.707, -0.707);
			\draw [fill=blue] (A) circle (2pt) node [left] {};
			\draw [fill=blue] (B) circle (2pt) node [right] {};
			\draw [fill=blue] (C) circle (2pt) node [right] {};
     		\draw [fill=blue] (D) circle (2pt) node [above] {};
            \draw [fill=blue] (E) circle (2pt) node [above] {};
            
			\draw [-latex, red, thick] (A) -- (B);
			\draw [-latex, red, thick] (A) -- (C);
            \draw [-latex, red, thick] (A) -- (D);
            \draw [-latex, red, thick] (A) -- (E);
			\end{tikzpicture}
				}
				\subfloat[]{
			\begin{tikzpicture}
			\coordinate (A) at (0,0);
			\coordinate (B) at (1,0);
			\coordinate (C) at (0.9397, -0.3420);
            \coordinate (D) at (0.3420,0.9397);
            \coordinate (E) at (0, 1);
			\draw [fill=blue] (A) circle (2pt) node [left] {};
			\draw [fill=blue] (B) circle (2pt) node [right] {};
			\draw [fill=blue] (C) circle (2pt) node [right] {};
     		\draw [fill=blue] (D) circle (2pt) node [above] {};
            \draw [fill=blue] (E) circle (2pt) node [above] {};
            
			\draw [-latex, red, thick] (A) -- (B);
			\draw [-latex, red, thick] (A) -- (C);
            \draw [-latex, red, thick] (A) -- (D);
            \draw [-latex, red, thick] (A) -- (E);
			\end{tikzpicture}
				}
	\caption{Tight frames of \( \R{2} \).}				
		\end{figure}
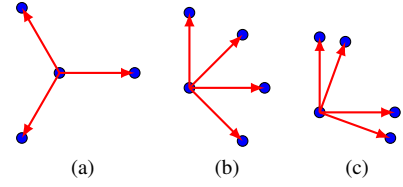
The collection of vectors pointing towards the corners of platonic solids also constitute tight frames; several other examples may be found in \cite{benedetto2003finite}.

The study of frames was started by Duffin and Schaeffer \cite{duffin1952class} and improvised greatly by work of Daubechies et al. in \cite{daubechies1986painless}. Recent work on characterization of tight frames as minimizers of a certain potential function was done in \cite{benedetto2003finite} and \cite{casazza2006physical}. The growing interest in developing the theory of tight frames is mainly because, tight frames possess several desirable properties. For example, the representation of signals using tight frames exhibits better resilience to noise and quantization \cite{shen2006image}, \cite{zhou2016adaptive}. Tight frames are also considered to be good for representing signals sparsely, and are the \( \ell_2 \)-optimal dictionaries for representing vectors that are uniformly distributed over spheres \cite[Proposition 2.13, p.\ 10]{sheriff2016optimal}. These are just a few among the plethora of useful properties of tight frames that make them relevant and important for many applications. 

We now formally discuss the concept of tight frames, their properties, and their characterization.
\begin{definition}
The fame operator \( G_{ \ftf } : H_n \longrightarrow H_n \) of a sequence \( \ftf \) of vectors is defined by
\[
G_{ \ftf } (v) \coloneqq \sum_{i = 1}^K \langle v, v_m \rangle v_m.
\]
\end{definition}

The following result provides the necessary and sufficient conditions for tightness of a frame.
\begin{theorem}\cite[Proposition 1, p.\ 2, ]{casazza2006physical} \cite[Theorem 2.1, p.\ 4]{benedetto2003finite}
\label{theorem:casazza}
A sequence of vectors \( \ftf \subset H_n \) is a tight frame to \( H_n \) with some constant \( \A > 0 \), if and only if 
\begin{itemize}
\item The lengths \( \norm{v_i} \) of the frame vectors satisfy
\begin{equation*}
\label{eq:nec-suff-condition-for-tightness}
\max_{i = 1,\ldots,K} \norm{v_i}^2 \leq \A \coloneqq \frac{1}{n} \sum_{i = 1}^K \norm{v_i}^2.
\end{equation*}

\item The frame operator satisfies \( G_{ \ftf } = \A I_n  \), where \( I_n \) is the identity operator on \( H_n \).
\end{itemize}
\end{theorem}

\subsection{Characterization of Tight Frames}
We observe that the definition of a tight frame does not give a direct method to compute them. One of the ways to compute tight frames is by minimizing a certain potential function that is motivated from physical examples briefly outlined in \cite{benedetto2003finite}, \cite{casazza2006physical}. The following definitions are relevant for this article:
 
\begin{definition}
The frame force \( \FF : H_n \times H_n \longrightarrow H_n \) experienced by a point \( u \in H_n \) due to another point \( v \in H_n \) is defined by
\[
\FF \left( u,v \right) \coloneqq 2 \; \langle u, v \rangle \;(u - v).
\]
\end{definition}

\begin{definition}
\label{def:FP}
The frame potential function \( \FP : H_n^K \longrightarrow \R{} \) of a sequence of vectors \( \ftf \) is defined by
\[
\FP \ftf \coloneqq \sum_{i = 1}^K \sum_{j = 1}^K \abs{ \langle v_j , v_i \rangle}^2.
\] 
\end{definition}

Observe that the frame force between any two points in \( H_n \) pushes them towards orthogonality. In this sense we can say that a tight frame is the ``most orthogonal'' collection of vectors in \( H_n \); alternatively, a tight frame is a collection of vectors that are maximally spread out in the space. It was established in \cite{casazza2006physical}, \cite{benedetto2003finite} that finite tight frames for \( H_n \) are in an equilibrium configuration under the frame force, and hence are the local minimizers of the associated frame potential. Thus, tight frames are characterized as the minimizers of the frame potential, and can be computed by variational techniques. This particular fact is the assertion of the following theorem.
\begin{theorem}\cite[Proposition 4, p.\ 8]{casazza2006physical}
\label{FP theorem}
For any sequence of vectors \( \ftf \subset H_n \), the following holds
\[
\frac{1}{n} \bigg( \sum_{i = 1}^K \norm{v_i}^2 \bigg)^2 \leq \FP\ftf.
\]
The preceding inequality is satisfied with equality if and only if \( \ftf \) is a tight frame of \( H_n \).
\end{theorem}

From Theorem \ref{FP theorem} we conclude that tight frames are not only local, but are global minimizers of the frame potential, subject to the constraint that the length of each vector in the sequence is fixed. Thus, we can say that the frame potential of a given sequence \( \ftf \) gives us a measure of its tightness. However, the lower bound of the frame potential given by Theorem \ref{FP theorem} depends on the sequence itself. This restricts the use of the frame potential of a sequence as a measure of its tightness when comparing two generic frames. For example, let us consider the frames \( F_1 \coloneqq \big( (1, 0)\transp, (1, 1)\transp \big) \) and \( F_2 \coloneqq \big( (10, 0)\transp, (-5, -5\sqrt{3})\transp, (-5, 5\sqrt{3} )\transp \big) \) of \( \R{2} \). From Definition \ref{def:FP}, the frame potential of the frames \( F_1 \) and \( F_2 \) are \( 7 \) and \( 45000 \) respectively. Even though \( 7 < 45000 \), from the second assertion of Theorem \ref{theorem:casazza} we see that the frame \( F_1 \) is not tight, where as \( F_2 \) is a tight frame. To address this issue, we normalize the frame potential by appropriately scaling the vectors so that the  frame potential calculated this way is a constant for every tight frame.

Let \( \ftf \subset H_n \) be any sequence of vectors, we define the \emph{Normalized Frame Potential} ( \( \MFP \) ) of the sequence \( \ftf \) using its frame operator as
\begin{equation}
\label{eq:MFP-definition}
\MFP\ftf \coloneqq \frac{ \trace G_{ \ftf }^2 }{ \Big( \trace G_{ \ftf } \Big)^2 }.
\end{equation}
We get the following analogue of Theorem \ref{FP theorem}.
\begin{proposition}
\label{theorem:MFP-theorem}
For any sequence of vectors \( \ftf \subset H_n \), we have
\[
\MFP\ftf \geq \frac{1}{n}.
\]
The preceding inequality is satisfied with the equality if and only if \( \ftf \) is a tight frame of \( H_n \).
\end{proposition}
\begin{proof}
From \cite[Lemma 1, p.\ 7]{casazza2008classes}
\begin{equation*}
\begin{aligned}
\trace G_{ \ftf }^2 &= \FP\ftf \quad \text{and it also follows that}, \\
\trace G_{ \ftf }   &= \sum_{i = 1}^K \norm{v_i}^2.
\end{aligned}
\end{equation*}
Then
\begin{equation*}
\MFP\ftf = \frac{ \FP\ftf }{ \Big( \sum_{i = 1}^K \norm{v_i}^2 \Big)^2 },
\end{equation*}
and the assertion follows immediately from Theorem \ref{FP theorem}.
\end{proof}

The \( \MFP \) of the frames \( F_1 \) and \( F_2 \) are \( 7/9 \) and \( 1/2 \) respectively, which clearly indicates that the frame \( F_2 \) is tighter than the frame \( F_1 \). Our \( \MFP \) \eqref{eq:MFP-definition} allows us to compare the tightness of any two arbitrary sequences of vectors in \( H_n \) and therefore can be regarded as a valid measure of tightness of the given sequence of vectors.

\section{Main Results}

We would like to find the optimal orientations of the columns of the matrix \( \kalman{A}{B}{T} \), that optimize the three measure of qualities discussed in Section \ref{section:Review-of-some-conventional-MOQ's-for-LTI-systems}, namely, \( \trace \big( \grammian{A}{B}{T}^{-1} \big) \), \( \lambda_{\min}^{-1}( \grammian{A}{B}{T} ) \) and \( \det(\grammian{A}{B}{T}) \). Let us consider a generic sequence \( \ftf \subset \R{n} \) of vectors. Then the three MOQ's stated above evaluated for the generic sequence \( \ftf \) are given in terms of the corresponding frame operator \( G_{ \ftf } \) by \( \trace \big( G_{ \ftf }^{-1} \big) \), \( \lambda_{\min}^{-1} \big( G_{ \ftf } \big) \) and \( \det \big( G_{ \ftf } \big) \) respectively. Let \( \big( \alpha_1, \ldots, \alpha_K \big) \) be a non increasing sequence of positive real numbers such that \( \alpha_1 \leq \frac{1}{n} \sum_{i = 1}^K \alpha_i \). 

In order to find the optimal orientation of the vectors, we optimize the three objective functions  \( \trace \big( G_{ \ftf }^{-1} \big) \), \( \lambda_{\min}^{-1} \big( G_{ \ftf } \big) \) and \( \det \big( G_{ \ftf } \big) \), subject to the constraint that the lengths of the vectors are fixed, i.e., \( \langle v_i, v_i \rangle = \alpha_i \) for all \( i = 1, \ldots, K \). Thus, we have the following three optimization problems:
\begin{equation}
\label{eq:trace-MOQ-optimization}
\begin{aligned}
& \underset{v_i}{\text{minimize}}
& & \trace \Big( G_{ ( v_1, \ldots, v_K ) }^{-1} \Big)  \\
& \text{subject to}
& & \langle v_i, v_i \rangle = \alpha_i \quad \text{for all \( i = 1, \ldots, K \)}, \\
& & & \Span\ftf = \R{n}.
\end{aligned}
\end{equation}

\begin{equation}
\label{eq:min-eigenvalue-MOQ-optimization}
\begin{aligned}
& \underset{v_i}{\text{minimize}}
& & \lambda_{\min}^{-1} \Big( G_{ \ftf } \Big)  \\
& \text{subject to}
& & \langle v_i, v_i \rangle = \alpha_i\quad \text{for all \( i = 1, \ldots, K \)}, \\
& & & \Span\ftf = \R{n}.
\end{aligned}
\end{equation}

\begin{equation}
\label{eq:det-MOQ-optimization}
\begin{aligned}
& \underset{v_i}{\text{maximize}}
& & \det \Big( G_{ \ftf } \Big)  \\
& \text{subject to}
& & \langle v_i, v_i \rangle = \alpha_i \quad \text{for all \( i = 1, \ldots, K \)}, \\
& & & \Span\ftf = \R{n}.
\end{aligned}
\end{equation}
Note that the optimization problems \eqref{eq:trace-MOQ-optimization}, \eqref{eq:min-eigenvalue-MOQ-optimization} and \eqref{eq:det-MOQ-optimization} are optimization problems over sequences of vectors. The next Theorem characterizes and provides solutions of \eqref{eq:trace-MOQ-optimization}, \eqref{eq:min-eigenvalue-MOQ-optimization} and \eqref{eq:det-MOQ-optimization}; in its proof we get equivalent optimization problems with non negative definite matrices as variables. 
\begin{theorem}
\label{theorem:optimiality-of-tight-frames-for-classical-MOQ's}
A sequence of vectors \( \ftf \), that is feasible for the optimization problems \eqref{eq:trace-MOQ-optimization}, \eqref{eq:min-eigenvalue-MOQ-optimization} and \eqref{eq:det-MOQ-optimization} is an optimal solution if and only if it is a tight frame of \( \R{n} \) . 
\end{theorem}
\begin{proof}
For the non increasing finite sequences \( \eigval{} \coloneqq (\eigval{1}, \ldots, \eigval{n}) \), and \( \alpha \coloneqq ( \alpha_1, \ldots, \alpha_K)  \) of positive real numbers we define the relation \( \eigval{} \succ  \alpha \) if the following two conditions hold:
\begin{equation}
\label{eq:rod-majorization-inequality}
\begin{aligned}
\sum_{i = 1}^m \eigval{i} & \geq \sum_{i = 1}^m \alpha_i \quad \text{for all \( m = 1, \ldots, n - 1 \), and} \\
\sum_{i = 1}^n \eigval{i} & = \sum_{i = 1}^K \alpha_i.
\end{aligned}
\end{equation}
The conditions of \eqref{eq:rod-majorization-inequality} are analogue of the standard majorization conditions \cite[Chapter 1]{marshall1979inequalities}.

For a symmetric non negative definite matrix \( G \in \possemdef{n} \), we define \( \eigval{}(G) \coloneqq \big( \eigval{1}(G), \ldots, \eigval{n}(G) \big) \) to be the non increasing sequence of the eigenvalues of the matrix \( G \). Let \( \A \coloneqq \big( \frac{1}{n} \sum_{i = 1}^K \alpha_i \big) \).

We will use the following result related to the decomposition of non-negative definite matrices. 
\begin{lemma}\cite[Theorem 4.6]{antezana2007schur} \cite[Theorem 2.8, p.\ 4]{casazza2008classes}
\label{lemma:rod}
For any given sequence of positive real numbers \( \alpha \) and a non negative definite matrix \( G \in \posdef{n} \) with \( K \geq n \), the following statements are equivalent:
\begin{itemize}
\item There exists a sequence of vectors \( \ftf \subset \R{n}\) such that \( G = G_{\ftf} \) and \( \langle v_i, v_i \rangle = \alpha_i \) for all \( i = 1, 2, \ldots, K \).

\item \( \eigval{}(G) \succ \alpha \).
\end{itemize}
\end{lemma}
On the one hand, for any sequence of vectors \( \ftf \) that are feasible for the optimization problems \eqref{eq:trace-MOQ-optimization}, \eqref{eq:min-eigenvalue-MOQ-optimization} and \eqref{eq:det-MOQ-optimization}, from Lemma \ref{lemma:rod} it follows that \( G_{ \ftf } \in \posdef{n} \) and \( \eigval{}(G) \succ \alpha \). On the other hand, for any symmetric positive definite matrix \( G \) such that \( \eigval{}(G) \succ \alpha \), we conclude again from Lemma \ref{lemma:rod} that there exists a sequence of vectors \( \ftf \subset \R{n} \) such that \( G = G_{ \ftf } \), \( \Span\ftf = \R{n} \), and \( \langle v_i, v_i \rangle = \alpha_i \) for all \( i = 1, \ldots, K \). Therefore, under the mapping 
\begin{equation}
\label{eq:rod-equivalence-map}
\R{n \times K} \ni \ftf \longmapsto G_{ \ftf } \in \possemdef{n}, 
\end{equation}
the optimization problems:
\begin{equation}
\label{eq:trace-MOQ-optimization-1}
\begin{aligned}
& \underset{G \; \in \; \posdef{n}}{\text{minimize}}
& & \trace ( G^{-1} )  \\
& \text{subject to}
& & \eigval{}(G) \succ \alpha ,
\end{aligned}
\end{equation}

\begin{equation}
\label{eq:min-eigenvalue-MOQ-optimization-1}
\begin{aligned}
& \underset{G \; \in \; \posdef{n}}{\text{minimize}}
& & \lambda_{\min}^{-1} (G)  \\
& \text{subject to}
& & \eigval{}(G) \succ \alpha ,
\end{aligned}
\end{equation}

\begin{equation}
\label{eq:det-MOQ-optimization-1}
\begin{aligned}
& \underset{G \; \in \; \posdef{n}}{\text{maximize}}
& & \det (G)  \\
& \text{subject to}
& & \eigval{}(G) \succ \alpha ,
\end{aligned}
\end{equation}
are equivalent to \eqref{eq:trace-MOQ-optimization}, \eqref{eq:min-eigenvalue-MOQ-optimization} and \eqref{eq:det-MOQ-optimization} respectively.

\begin{claim}
\label{claim:trace-moq}
\( G\opt \coloneqq \A I_n \) is the unique (corresponding to a given \( \alpha \)) optimal solution of the optimization problem \eqref{eq:trace-MOQ-optimization-1}. 
\end{claim}
\begin{proof}
Let us consider the following optimization problem:
\begin{equation}
\label{eq:trace-MOQ-optimization-different}
\begin{aligned}
& \underset{G \; \in \; \posdef{n}}{\text{minimize}}
& & \trace ( G^{-1} )  \\
& \text{subject to}
& & \trace(G) = n \A.
\end{aligned}
\end{equation}
The optimization problem \eqref{eq:trace-MOQ-optimization-different} is the same as (22) in \cite[p.\ 17]{sheriff2016optimal} with \( \Sigma_V = I_n \), whose solution is given in (25) of the same article. Therefore, from \cite{sheriff2016optimal} we conclude that \( G\opt \coloneqq \A I_n \) is the unique optimal solution to the problem \eqref{eq:trace-MOQ-optimization-different}. 

It is easy to see that for a symmetric non negative definite matrix \( G \), the condition that \( \eigval{}(G) \succ \alpha \) is sufficient for the equality \( \trace(G) = n \A \) to hold. Therefore, the optimum value \( \trace( {G\opt}^{-1} ) \) in the optimization problem \eqref{eq:trace-MOQ-optimization-different} is a lower bound of the optimal value (if it exists) of the problem \eqref{eq:trace-MOQ-optimization-1}.

However, it can be easily verified that \( \eigval{}( G\opt ) \succ \alpha \); thus, \( G\opt \) is feasible for the optimization problem \eqref{eq:trace-MOQ-optimization-1}, and the objective function evaluated at \( G = G\opt \) is equal to \( \trace( {G\opt}^{-1} ) \), which is a lower bound of the optimal value of \eqref{eq:trace-MOQ-optimization-1}. Therefore, the optimization problem \eqref{eq:trace-MOQ-optimization-1} admits an optimal solution. Since \( G\opt \) is the unique solution to \eqref{eq:trace-MOQ-optimization-different}, we conclude that \( G\opt \) is the unique optimal solution to \eqref{eq:trace-MOQ-optimization-1} as well.
\end{proof}

\begin{claim}
\label{claim:min-eigen-value}
\( G\opt \coloneqq \A I_n \) is the unique (corresponding to a given \( \alpha \)) optimal solution of the optimization problem \eqref{eq:min-eigenvalue-MOQ-optimization-1}. 
\end{claim}
\begin{proof}
The objective function and the constraint in problem \eqref{eq:min-eigenvalue-MOQ-optimization-1} can be explicitly characterized in terms of the eigenvalues of the variable matrix \( G \). Then under the map 
\begin{equation}
\label{eq:matrix-to-eigenvalues-map}
\posdef{n} \ni G \longmapsto  \eigval{}(G),
\end{equation} 
we get the following optimization problem equivalent to \eqref{eq:min-eigenvalue-MOQ-optimization-1}.
\begin{equation}
\label{eq:min-eigenvalue-MOQ-optimization-2}
\begin{aligned}
& \underset{\eigval{i} > 0}{\text{minimize}}
& & \big( \min \{ \eigval{1}, \ldots, \eigval{n} \} \big)^{-1}  \\
& \text{subject to}
& & \eigval{} \succ \alpha. 
\end{aligned}
\end{equation}
For every sequence \(  ( \eigval{1}, \ldots, \eigval{n} ) \) that is feasible for the optimization problem \eqref{eq:min-eigenvalue-MOQ-optimization-2}, we see that  \( \sum_{i = 1}^n \eigval{i} = \sum_{i = 1}^K \alpha_i = n \A \). Therefore,
\begin{equation*}
\begin{aligned}
\min \{ \eigval{1}, \ldots, \eigval{n} \} & \leq \A, \quad \text{leading to,} \\
\big( \min \{ \eigval{1}, \ldots, \eigval{n} \} \big)^{-1} & \geq { \A }^{-1}.
\end{aligned}
\end{equation*}

Therefore, the value \( { \A }^{-1} \) is a lower bound of the optimal value (if it exists) of the optimization problem \eqref{eq:min-eigenvalue-MOQ-optimization-2}. Let us define \( \eigval{i}\opt \coloneqq \A \) for all \( i = 1,2,\ldots,n \). It can be easily verified that 
\begin{equation}
\label{eq:min-eigval-solution-2}
\begin{aligned}
( \eigval{1}\opt, \ldots, \eigval{n}\opt ) & \succ \alpha \quad \text{and} \\
\big( \min \{ \eigval{1}\opt, \ldots, \eigval{n}\opt \} \big)^{-1} & = { \A }^{-1}.
\end{aligned}
\end{equation}

From \eqref{eq:min-eigval-solution-2} it follows that the sequence \(  ( \eigval{1}\opt, \ldots, \eigval{n}\opt ) \) is feasible for the optimization problem \eqref{eq:min-eigenvalue-MOQ-optimization-2}, and that the objective function evaluated at \( ( \eigval{1}\opt, \ldots, \eigval{n}\opt ) \) is equal to the lower bound \( { \A }^{-1} \). Therefore, we conclude that the optimization problem \eqref{eq:min-eigenvalue-MOQ-optimization-2} admits an optimal solution, and that the sequence \( ( \eigval{1}\opt, \ldots, \eigval{n}\opt ) \) is an optimal solution. It should also be noted that \( ( \eigval{1}\opt, \ldots, \eigval{n}\opt ) \) is the unique sequence of positive real numbers that satisfies \eqref{eq:min-eigval-solution-2}. Therefore, the sequence \( ( \eigval{1}\opt, \ldots, \eigval{n}\opt ) \) is the unique optimal solution of the problem \eqref{eq:min-eigenvalue-MOQ-optimization-2}. 

Due to the equivalence of the optimization problems \eqref{eq:min-eigenvalue-MOQ-optimization-1} and \eqref{eq:min-eigenvalue-MOQ-optimization-2}, we conclude that the optimization problem \eqref{eq:min-eigenvalue-MOQ-optimization-1} also admits an optimal solution. A non negative definite symmetric matrix \( G\opt \) is an optimal solution of the problem \eqref{eq:min-eigenvalue-MOQ-optimization-1} if and only if it satisfies
\begin{equation}
\label{eq:min-eigval-solution-1}
\eigval{i} (G\opt) = \A \quad \text{for all \( i = 1, 2, \ldots, n \),}
\end{equation}
and we know for a fact that \( \A I_n \) is the only matrix that satisfies \eqref{eq:min-eigval-solution-1} that is also feasible for the problem \eqref{eq:min-eigenvalue-MOQ-optimization-1}. Thus, \( G\opt = \A I_n \) is the unique optimal solution of the problem \eqref{eq:min-eigenvalue-MOQ-optimization-1}. 
 \end{proof} 

\begin{claim}
\label{claim:det-moq}
\( G\opt \coloneqq \A I_n \) is the unique (corresponding to a given \( \alpha \)) optimal solution to the optimization problem \eqref{eq:det-MOQ-optimization-1}. 
\end{claim}
\begin{proof}
Writing the objective function and the constraints of the optimization problem \eqref{eq:det-MOQ-optimization-1} in terms of the eigenvalues \(  \eigval{}(G) \) of the matrix \( G \), we get the following optimization problem equivalent to \eqref{eq:det-MOQ-optimization-1}.
\begin{equation}
\label{eq:det-MOQ-optimization-2}
\begin{aligned}
& \underset{\eigval{i} > 0}{\text{maximize}}
& & \prod_{i = 1}^n \eigval{i}  \\
& \text{subject to}
& & \eigval{} \succ \alpha.
\end{aligned}
\end{equation}
For every sequence \( ( \eigval{1}, \ldots, \eigval{n} ) \) that is feasible for the problem \eqref{eq:det-MOQ-optimization-2}, it follows that 
\begin{equation}
\prod_{i = 1}^n \eigval{i} \leq \Big( \frac{1}{n} \sum\limits_{i = 1}^n \eigval{i} \Big)^n = \Big( \frac{1}{n} \sum\limits_{i = 1}^K \alpha_i \Big)^n = { \A }^n .
\end{equation}
Therefore, the value \( { \A }^n \) is an upper bound to the optimal value (if it exists) of the optimization problem \eqref{eq:det-MOQ-optimization-2}. Let us define \( \eigval{i}\opt \coloneqq \A \) for all \( i = 1,2,\ldots,n \), it follows that  
\begin{equation}
\label{eq:det-eigval-solution-2}
\begin{aligned}
( \eigval{1}\opt, \ldots, \eigval{n}\opt ) & \succ \alpha, \ \text{and} \\
\prod_{i = 1}^n \eigval{i}\opt & = { \A }^n .
\end{aligned}
\end{equation}
It should also be noted that \( ( \eigval{1}\opt, \ldots, \eigval{n}\opt ) \) is the unique sequence of positive real numbers that satisfies \eqref{eq:det-eigval-solution-2}. Therefore, the optimization problem \eqref{eq:det-MOQ-optimization-2} admits an unique optimal solution and the sequence \( ( \eigval{1}\opt, \ldots, \eigval{n}\opt ) \) is the unique optimizer. 

Equivalently, as we have seen in the proof of Claim \ref{claim:min-eigen-value}, \( G\opt \coloneqq \A I_n \) is the unique optimal solution to the optimization problem \eqref{eq:det-MOQ-optimization-1}.
\end{proof}
Due to the equivalence of the optimization problem pairs \eqref{eq:trace-MOQ-optimization}-\eqref{eq:trace-MOQ-optimization-1}, \eqref{eq:min-eigenvalue-MOQ-optimization}-\eqref{eq:min-eigenvalue-MOQ-optimization-1} and \eqref{eq:det-MOQ-optimization}-\eqref{eq:det-MOQ-optimization-1}, we conclude from the Claims \ref{claim:trace-moq}, \ref{claim:min-eigen-value}, and \ref{claim:det-moq} that a sequence \( \ftf \) of vectors is an optimal solution if and only if its frame operator satisfies \( G_{ \ftf } = \A I_n \), which in turn is true if and only if \( \ftf \) is a tight frame (from Theorem \ref{theorem:casazza}). 
\end{proof}

In addition to the fact that tight frames are optimal solutions of the optimization problems \eqref{eq:trace-MOQ-optimization}, \eqref{eq:min-eigenvalue-MOQ-optimization} and \eqref{eq:det-MOQ-optimization}, we have previously listed some desirable properties of tight frames for representation of generic vectors in Section \ref{sec:tight-frame}. Most of the reachability(both ballistic and servomechanism \cite[p.\ 74-75]{brockett2015finite}) control problems arising in practice, involve solving the following linear equation for various values of \( x_0 \) and \( x_T \):
\begin{equation}
\label{eq:reachability-problem}
\begin{aligned}
x_T - A^T x_0 &= \sum_{i = 0}^{T -1} A^i B \; u(T-1-i)
\end{aligned}
\end{equation}
either directly or indirectly. This is equivalent to expressing the vector \( x_T - A^T x_0 \) as a linear combination of the columns of \( \kalman{A}{B}{T} \). We would like to have the the columns of \( \kalman{A}{B}{T} \) to be oriented as tightly as possible in order to inherit those advantages. Thus, a measure of quality of the LTI system \eqref{eq:lti-system} should be the measure of tightness of the columns of the matrix \( \kalman{A}{B}{T} \); this is a key proposal of this article.

\begin{definition}
We define a measure of quality \( \moq{A}{B}{T} \) for the LTI system \eqref{eq:lti-system} as a measure of the tightness of the columns of \( \kalman{A}{B}{T} \); defined by 
\begin{equation}
\label{eq:moq-definition}
\moq{A}{B}{T} \coloneqq \frac{\trace \grammian{A}{B}{T}^2}{\left( \trace \grammian{A}{B}{T} \right)^2}.
\end{equation}
\end{definition}

 From Proposition \ref{theorem:MFP-theorem}, for any LTI system given by the pair \( (A,B) \in (\R{n \times n} \times \R{n \times m} ) \) and \( \mathbb{N} \ni T \geq \tau \), we have 
\begin{equation}
\label{eq:moq-inequality}
\moq{A}{B}{T} \geq 1/n, 
\end{equation}
and equality holds only when the columns of matrix \( \kalman{A}{B}{T} \) constitute a tight frame for \( \R{n} \). A direct application of the Proposition \ref{theorem:MFP-theorem} gives us the following sufficient condition for the controllability of the LTI system in terms of the MOQ \( \moq{A}{B}{T} \).
\begin{proposition}
For the LTI system \eqref{eq:lti-system}, if \( \moq{A}{B}{T} < \frac{1}{n - 1} \) for some \( T \in \mathbb{N} \), then the LTI system is controllable in the classical sense.
\end{proposition}
\begin{proof}
Suppose not, i.e., the system \( (A,B) \) is not controllable but satisfies \( \moq{A}{B}{T} < \frac{1}{n - 1} \) for some \( T \in \mathbb{N} \). Let \( v_i \) be the \( i^{\text{th}} \) column of the matrix \( \kalman{A}{B}{T} \) for \( i = 1, \ldots, mT \). Then \( \Span(v_1, \ldots, v_{mT} ) \) is a \( d \) dimensional subspace of \( \R{n} \), where \( d \leq n - 1 \). In particular, each \( v_i \) belongs to this subspace, which is a Hilbert space in its own right. From Proposition \ref{theorem:MFP-theorem} we conclude that 
\[
\moq{A}{B}{T} = \MFP (v_1, \ldots, v_{mT}) \geq \frac{1}{d} \geq \frac{1}{n - 1} ,
\]
which is a contradiction. The assertion follows.
\end{proof}


\bibliographystyle{IEEEtran}
\bibliography{ref}

\end{document}